\documentclass[11pt,leqno]{amsart}
\usepackage{amsmath, bbm, amsthm, amssymb, amsfonts, enumerate, graphicx}
\usepackage[bookmarks=false]{hyperref}

\theoremstyle{plain}
\newtheorem{thm}{Theorem}
\newtheorem{lemma}[thm]{Lemma}
\newtheorem{cor}[thm]{Corollary}
\newtheorem{prop}[thm]{Proposition}

\theoremstyle{definition}
\newtheorem{defi}[thm]{Definition}

\newtheorem{remark}[thm]{Remark}

\newtheorem{claim}[thm]{Claim}

\numberwithin{thm}{section}
\numberwithin{equation}{section}

\DeclareMathOperator{\dcc}{d_{\rm cc}}          
\DeclareMathOperator{\ellcc}{\ell_{cc}}     
\newcommand{\Ball}{B_{\rm cc}}          
\newcommand{\ABox}{\mathcal{B}}
\DeclareMathOperator{\dhh}{d_{\rm K}}           

\newcommand{\comp}[1]{#1^c}
\newcommand{\dpi}[1]{\operatorname{d}_{\pi_t}^{#1}}
\newcommand{\interior}{\operatorname{int}}
\newcommand{\dccOm}[1]{\operatorname{d}_{\rm cc}^{#1}}

\newcommand{\R}{\mathbb R}

\newcommand{\Heis}{\mathbb{H}}

\providecommand{\norm}[1]{\lVert#1\rVert}       

\newcommand{\comment}[1]{}

\title{Quasiconvexity in the Heisenberg group} 
\author{David A. Herron}
\address{Department of Mathematical Sciences \\ University of Cincinnati \\ French Hall West \\ Cincinnati, OH 45221-0025}
\email{David.Herron@math.UC.edu}
\author{Anton Lukyanenko}
\address{Department of Mathematics \\ University of Michigan \\ 530 Church Street, Ann Arbor, MI 48109}
\email{anton@lukyenenko.net}
\author{Jeremy T. Tyson}
\address{Department of Mathematics \\ University of Illinois \\ 1409 West Green St. \\ Urbana, IL, 61801}
\email{tyson@illinois.edu}
\date{\today}
\thanks{J.T.T. was supported by Simons Foundation Collaboration Grant 353627 `Geometric Analysis in Sub-Riemannian and Metric Spaces'.  D.A.H.\ was  supported by the Charles Phelps Taft Research Center.  A.L. was supported by NSF RTG grant  DMS-1045119.  All three authors were  supported by the NSF grant DMS-1500454.}

\begin{document}
\dedicatory{Dedicated to William Goldman on the occasion of his 60th birthday.}

\begin{abstract}
We show that if $A$ is a closed subset of the Heisenberg group whose vertical projections 
are nowhere dense, then the complement of $A$ is quasiconvex. In particular, closed sets which are null sets for the cc-Hausdorff $3$-measure have quasiconvex complements. Conversely, we exhibit a compact totally disconnected set of Hausdorff dimension three whose complement is not quasiconvex.
\end{abstract}

\maketitle

\section{Introduction}

A metric space $(X,d)$ is {\it $c$-quasiconvex}, $c \ge 1$, if every pair of points $x,y$ in $X$ can be joined by a rectifiable curve whose length is no more than $c \, d(x,y)$.

Quasiconvex spaces arise systematically in the theory of analysis in metric spaces. A folklore result attributed to Semmes (see, for instance, \cite[\S 17]{ch:differentiability}) asserts that every doubling metric measure space supporting a Poincar\'e inequality is quasiconvex. Also, all uniform domains and spaces are quasiconvex. On the other hand, every quasiconvex space is bi-Lipschitz equivalent to a length space. Since most of the machinery of analysis in metric spaces is invariant under bi-Lipschitz transformations, and since it is often convenient to work in the length space setting, quasiconvexity is an extremely relevant and useful assumption.

It is thus an interesting problem to determine which spaces are quasiconvex. In this paper we address this question for open subsets of the Heisenberg group $\Heis$ equipped with a sub-Riemannian metric. Our results parallel those obtained by Hakobyan and the second author \cite{hh:quasiconvexity} for open subsets of Euclidean space. In particular, we give both topological and measure-theoretic conditions on a closed set $A \subset \Heis$ which ensure that $\Heis \setminus A$ is quasiconvex.

Let us recall that the Heisenberg group $\Heis$ is the space $\R^3$ equipped with the nonabelian group law
$$  (x,y,t)*(x',y',t') := (x+x',y+y',t+t'+2(x'y-xy'))$$
and a natural left-invariant \emph{Carnot--Carath\'eodory} path metric $\dcc$ (see \S \ref{sec:summaryofproof} for the definition). The space $(\Heis,\dcc)$ is complete and geodesic.
Topologically, this space is homeomorphic to $\R^3$ itself, however, it is not bi-Lipschitz equivalent to $\R^3$. In fact, the Hausdorff dimension of $(\Heis,\dcc)$ is $4$, and the identity map from $\Heis$ to $\R^3$ is locally Lipschitz with a locally $\tfrac12$-H\"older continuous inverse. Thus, for instance, the Hausdorff dimension of a subset of $\Heis$ in the cc-metric and in the underlying Euclidean metric need not agree. The precise relationship between these two notions of Hausdorff dimension is well understood: this is the content of the so-called {\it Dimension Comparison Theorem}. See Theorem \ref{DCT}.

While $\R^n$ contains plenty of convex domains, convexity in the Heisenberg group is a problematic concept.
For instance, Monti and Rickly \cite{mr:convex} showed that any nonempty set $A$ in $\Heis$ that is geodesically convex (that is, $A$ contains all geodesic segments joining any pair of points in $A$) is either a point, a connected subset of a geodesic, or all of $\Heis$. Alternate approaches to convexity in $\Heis$ have been considered by Lu, Manfredi and Stroffolini \cite{lms:convex} and Danielli, Garofalo and Nhieu \cite{dgn:convex}.

Due to the uniqueness of geodesics in Euclidean space, the notions of convexity and $1$-quasiconvexity are equivalent in that setting. It is well-known that geodesics in $\Heis$ are not unique, even locally. Put another way, the injectivity radius of $\Heis$ equals zero at every point. As a result, geodesic convexity and $1$-quasiconvexity need not coincide in this setting.


Many sets in $\Heis$ are quasiconvex. For example, it is easy to show that any half-space is quasiconvex (indeed, vertical half-spaces are even $1$-quasiconvex), and it follows from \cite[Theorem 1.3]{MR2135732} that any compact domain in $\Heis$ with $C^{1,1}$ boundary is quasiconvex and has a quasiconvex complement (see \cite{Fassler-Lukyanenko-Tyson} for an application of this fact in sub-Riemannian geometric mapping theory).

The semi-direct product decomposition of $\Heis$ as the product of a vertical homogeneous subgroup and a horizontal homogeneous subgroup induces projection mappings into each factor. For precise definitions of these projections, see Subsection \ref{subsec:projections}. We now state our first main theorem.

\begin{thm}\label{thm:main}
Let $A\subset \Heis$ be a closed set. Assume that the vertical projections of $A$ into the $xt$- and $yt$-planes are nowhere dense. Then $\Heis \setminus A$ is quasiconvex.
\end{thm}

In particular, if the vertical projections of a closed set $A$ have cc-Hausdorff $3$-measure zero, then the conclusion of Theorem \ref{thm:main} continues to hold. Note that vertical planes in $\Heis$ have cc-Hausdorff dimension $3$, and the cc-Hausdorff $3$-measure on such a plane coincides (up to a constant multiple) with the standard Euclidean surface area measure. Since vertical projection mappings preserve null sets for the $3$-measure, we derive the following corollary.

\begin{cor}\label{cor:main}
Let $A\subset \Heis$ be a closed set of cc-Hausdorff $3$-measure zero. Then $\Heis \setminus A$ is quasiconvex.
\end{cor}

Corollary \ref{cor:main} can also be provided directly by appealing to the ACL characterization of Sobolev functions and the geometry of Sobolev removable sets. See Remark \ref{rem:cor-main-alternate-proof} for details.

In particular, sets of cc-Hausdorff dimension strictly less than $3$ have quasiconvex complements. Our next result shows that this observation is sharp.

\begin{thm}\label{thm:main-example}
There exists a compact and totally disconnected set $A\subset \Heis$ of cc-Hausdorff dimension $3$ whose complement is not quasiconvex.
\end{thm}

The Dimension Comparison Theorem in $\Heis$ allows us to convert assumptions regarding cc-Hausdorff dimension or measure into asumptions regarding Euclidean dimension or measure. The full statement of the Dimension Comparison Theorem in $\Heis$ appears in Subsection \ref{subsec:DCT}. For now, we observe that null sets for the Euclidean Hausdorff $2$-measure are necessarily null sets for the cc-Hausdorff $3$-measure. Hence Theorem \ref{thm:main} applies to closed null sets for the Euclidean Hausdorff $2$-measure. The example which we provide in the proof of Theorem \ref{thm:main-example} has Euclidean Hausdorff dimension equal to $2$.

This paper is organized as follows. In Section \ref{sec:background} we provide some additional background regarding analysis and geometry in the sub-Riemannian Heisenberg group. In particular, we define horizontal curves and the Carnot--Carath\'eodory metric, discuss the isometries and similarities of $(\Heis,\dcc)$, state the precise form of the Dimension Comparison Theorem, and define and discuss vertical projection mappings. We conclude Section \ref{sec:background} with short summaries of the proofs of Theorems \ref{thm:main} and \ref{thm:main-example}. In Sections \ref{S:pf} and \ref{S:ex} we give detailed proofs of these two theorems.

\

\paragraph{\bf Acknowledgements:} Research for this paper was conducted during visits of various subsets of the authors to the University of Illinois and the University of Cincinnati. The hospitality of these institutions is appreciated.

\section{Background}\label{sec:background}

In this section we briefly recall some well-known properties of the Heisenberg group which will be used in this paper.

\subsection{Horizontal curves}\label{subsec:horizontal}

Recall that the Heisenberg group is equipped with a pair of left-invariant vector fields
$$X=\frac\partial{\partial x} + 2y \frac\partial{\partial t} \quad\text{and}\quad Y=\frac\partial{\partial y} - 2x \frac\partial{\partial t}.$$
An absolutely continuous curve $\gamma\subset \Heis$ is \emph{horizontal} if its derivative lies almost everywhere in the span of $X$ and $Y$, i.e., for almost every $s$ one has $\gamma'(s)=a(s) X+b(s) Y$ for some \emph{control functions} $a$ and $b$. The \emph{Carnot--Carath\'eodory length (cc-length)} of a horizontal curve $\gamma$ is
$$
\ellcc(\gamma) = \int_\gamma \sqrt{a(s)^2+b(s)^2} \, ds \, ,
$$
which coincides with the Euclidean length of the $\pi_t$-projection of that curve into the $xy$-plane, where $\pi_t(x,y,t):=(x,y)$.

In this paper we use particularly simple horizontal curves that follow only one of the two vector fields $X$ and $Y$ at any given time. We refer to a compact horizontal curve $\gamma$ with endpoints $h_1, h_2$ as an \emph{$X$-line segment} if $b\equiv 0$ and a \emph{$Y$-line segment} if $a\equiv 0$. Clearly, such a $\gamma$ is in fact a Euclidean line segment, and we denote it by $[h_1, h_2]$.

A compact curve $\gamma$ is said to be a \emph{bang-bang curve} (the terminology is borrowed from control theory) if it can be written in the form
$$
\gamma = [h_1, h_2]\cup[h_2, h_3]\cup \cdots \cup [h_{n-1}, h_n],
$$
where each $[h_i, h_{i+1}]$ is either an $X$-line segment or a $Y$-line segment.

\subsection{The Carnot--Carath\'eodory distance and geodesics}\label{subsec:cc-distance}

The \emph{Carnot--Carath\'eodory distance (cc-distance)} between two points $p,q$ of $\Heis$, denoted $\dcc(p,q)$, is the infimum of the cc-lengths of all horizontal curves joining $p$ to $q$. It is well known that
\begin{enumerate}
\item[(i)] any two points of $\Heis$ can be joined by a cc-geodesic,
\item[(ii)] if $p$ and $q$ are not vertically separated (i.e., $\pi_t(p) \ne \pi_t(q)$), then there is a unique cc-geodesic joining $p$ to $q$, and
\item[(iii)] if $p=(x,y,t)$ and $q=(x,y,t')$ are vertically separated, then there is a one-parameter family of cc-geodesics joining $p$ to $q$, any two of which are related by a rigid Euclidean rotation about the vertical line through $p$ and $q$. The projection of any one of these geodesics into the $xy$-plane is a circle passing through $(x,y)$, whose area is equal to $\tfrac14|t'-t|$.
\end{enumerate}
It follows from (iii) and the definition of cc-length that the cc-distance between $p=(x,y,t)$ and $q=(x,y,t')$ is
$$
2\pi\sqrt{\frac{|t'-t|}{4\pi}} = \sqrt\pi|t'-t|^{1/2}.
$$

\subsection{The induced path metric on a domain in $\Heis$}

For a domain $\Omega$ in $\Heis$, the induced path metric $\dccOm{\Omega}$ is defined as follows: $\dccOm{\Omega}(p,q)$ is the infimum of the cc-lengths of all curves joining $p$ to $q$ in $\Omega$. Clearly, $\dccOm{\Omega}(p,q) \ge \dcc(p,q)$ for any $p$ and $q$. Quasiconvexity of $\Omega$ with quasiconvexity constant $C$ implies that $\dccOm{\Omega}$ is $C$-bi-Lipschitz equivalent to the restriction of $\dcc$ to $\Omega$. Conversely, if $\dccOm{\Omega}$ is $C$-bi-Lipschitz equivalent to $\dcc|_{\Omega\times\Omega}$, then $\Omega$ is $L$-quasiconvex for all $L>C$.

\subsection{Isometries and similarities of the Heisenberg group}

Isometries of the metric space $(\Heis,\dcc)$ include left translations
$$
h \mapsto h_0 * h, \qquad h_0 \in \Heis,
$$
and rotations
$$
(x,y,t) \mapsto (x\cos\theta-y\sin\theta,x\sin\theta+y\cos\theta,t), \qquad \theta \in \R.
$$
Moreover, for $r>0$, the group isomorphism
$$
\delta_r(x, y, t) = (rx, ry, r^2t)
$$
acts as a similarity of $\dcc$:
$$
\dcc( \delta_r h_1, \delta_r h_2) = r \dcc(h_1, h_2) \qquad h_1,h_2 \in \Heis\,.
$$

\subsection{The Dimension Comparison Theorem}\label{subsec:DCT}

The relationship between the cc-Hausdorff dimension and the Euclidean Hausdorff dimension of a subset of $\Heis$ is governed by the Dimension Comparison Theorem. The following theorem was proved by Balogh, Rickly and Serra-Cassano \cite{brsc:comparison}, see also \cite{MR2466427} for the corresponding statement in general Carnot groups.

\begin{thm}[Dimension Comparison Theorem in the Heisenberg group]\label{DCT}
Let $S \subset \Heis$. Then one has the sharp estimate $\beta_-(\dim_E S) \le \dim_{cc} S \le \beta_+(\dim_E S)$, where $\beta_-(\alpha) = \max\{\alpha,2\alpha-2\}$ and $\beta_+(\alpha) = \min\{2\alpha,\alpha+1\}$.
\end{thm}

Figure \ref{fig:DCT} illustrates the set of allowed dimension pairs $(\alpha,\beta)$, $\alpha = \dim_E S$, $\beta = \dim_{cc} S$, for subsets $S$ of the Heisenberg group.

\begin{figure}[h]
\includegraphics[width=.55\textwidth]{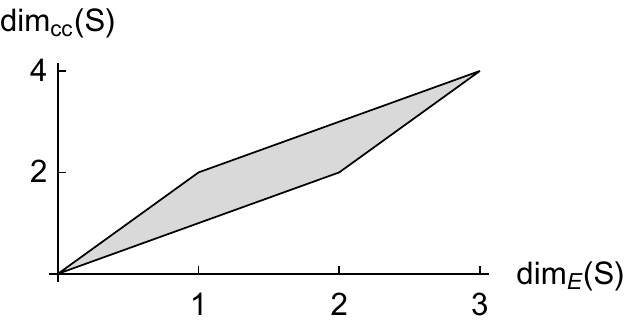}\caption{The Dimension Comparison Theorem in $\Heis$}\label{fig:DCT}
\end{figure}

In this paper, we will only use the following case of Theorem \ref{DCT}: \emph{If $S \subset \Heis$ has $\dim_E S \le 2$, then $\dim_{cc} S \le 3$.} In fact, if $S$ is a null set for Euclidean Hausdorff $2$-measure, then $S$ is a null set for cc-Hausdorff $3$-measure.

\subsection{Horizontal and vertical projection mappings}\label{subsec:projections}

The Heisenberg group admits two semidirect product decompositions, one associated to the splitting of $\Heis$ in terms of the $x$-axis and the $yt$-plane, and the other to the splitting of $\Heis$ in terms of the $y$-axis and the $xt$-plane. These two decompositions induce \emph{horizontal} natural projection maps $\Heis\xrightarrow{\pi_x,\pi_y}\Heis$, defined as follows:
\begin{align*}
  \pi_x(x,y,t)&:=(x,y,t)*(-x,0,0)=(0,y,t-2xy)\,, \\
  \pi_y(x,y,t)&:=(x,y,t)*(0,-y,0)=(x,0,t+2xy)\,.
\end{align*}

\begin{remark}
Recall that we also have the mapping $\pi_t:\Heis\rightarrow \R^2$ given by $\pi_t(x,y,t)=(x,y)$. The related mapping $(x,y,t)\mapsto (x,y,0)$ is a \emph{vertical projection}, which we will not use. 
\end{remark}
It is easy to see that a curve $\gamma\subset \Heis$ is an $X$-line segment (or $Y$-line segment) if and only if $\pi_x(\gamma)$ (respectively, $\pi_y(\gamma)$) is a single point. Although the maps $\pi_x$ and $\pi_y$ are not Lipschitz (when considered as maps from $(\Heis,\dcc)$ to either the $xt$- or $yt$-plane equipped with the cc-metric), it is nevertheless the case that $\pi_x$ and $\pi_y$ preserve null sets for the cc-Hausdorff $3$-measure; see \cite[Lemma 3.6]{cfo:betas} following from \cite{MR3511465}, and \cite[Proposition 4.3]{bcfmt:projections} for a weaker statement involving Hausdorff dimension.

\subsection{Summary of the proofs} \label{sec:summaryofproof}

To prove Theorem \ref{thm:main}, we start by showing that any two points $p,q$ in $\Heis$ can be joined by a bang-bang curve consisting of at most four segments and of length comparable to $\dcc(p,q)$. Assuming in addition that $p,q$ are in the complement of a closed set $A$, we then perturb each of the four segments so that it lies in the complement of $A$. As the line segments are perturbed, their endpoints no longer match up, but for sufficiently small perturbations the fact that $A$ is closed allows us to reconnect the corresponding endpoints by short geodesics, yielding the desired path between $p$ and $q$ in $\Heis \setminus A$.

The existence of perturbed segments lying in $\Heis \setminus A$ follows from the connection between $X$-line and $Y$-line segments and the projection mappings $\pi_x$ and $\pi_y$, and the assumption that $\pi_x(A)$ and $\pi_y(A)$ are nowhere dense.

\

Our proof of Theorem \ref{thm:main-example} is closely related to a similar example constructed for the Euclidean antecedent of this paper, \cite{hh:quasiconvexity}. We first show that it suffices to prove that for each $n$ there exists a set $A_n$ with the following properties.
\begin{enumerate}
\item $\Heis \setminus A_n$ is not $n$-quasi-convex.
\item $A_n$ is compact and totally disconnected.
\item The Euclidean Hausdorff dimension of $A_n$ is at least $2$.
\end{enumerate}
Note that here we switch already from the Heisenberg metric to the Euclidean metric to measure the dimension of $A_n$. We extend this philosophy further by modifying the metric involved in the definition of quasiconvexity.  Namely, rather than measuring the length of a curve $\gamma$ directly in $\Heis$, we first project $\gamma$ to $\R^2$ via the map $\pi_t$ and then measure the $\ell^1$ length of $\pi_t(\gamma)$. We refer to the resulting number as the \emph{$\pi_t$-length} of $\gamma$, and we consider the induced path pseudometrics $\dpi{}$ on $\Heis$ and $\dpi{\Omega}$ on $\Omega \subset \Heis$.

\begin{remark}
While for cc-rectifiable curves $\gamma$ the cc-length and the $\pi_t$-length are equivalent (up to a fixed multiplicative constant), the $\pi_t$-length is also finite for many cc-nonrectifiable curves. This fact allows us to establish the desired result without needing to maintain cc-rectifiability.
\end{remark}

We show that the complement of
$$
\ABox_0:=[-10n,10n]\times[-10n,10n]\times[0,1]
$$
is not $n$-convex by measuring $\dpi{\comp \ABox_0}(h^-, h^+)$, the $\pi_t$-distance in the complement of $\ABox_0$ between the points $h^+=(0,0,1)$ and $h^-=(0,0,0)$. We then exhibit a subset $A_n \subset \ABox_0$ such that the following is true.
\begin{enumerate}
\item $\dpi{\comp A_n}(h^+, h^-) \geq \dpi{\comp \ABox_0}(h^+, h^-)$.
\item The Euclidean Hausdorff dimension of $A_n$ is $2$.
\item The set $A_n$ is compact and totally disconnected.
\end{enumerate}

In fact, such a set was already constructed in \cite{hh:quasiconvexity}, and we provide a sketch of the construction for the reader's convenience.


\section{Closed sets with nowhere dense vertical projections have quasiconvex complement} \label{S:pf} 

The goal of this section is to prove Theorem \ref{thm:main}. We start by showing that any two points in $\Heis$ are connected by a bang-bang path of controlled length, consisting of at most four segments.

\begin{figure}[h]
\includegraphics[width=3in]{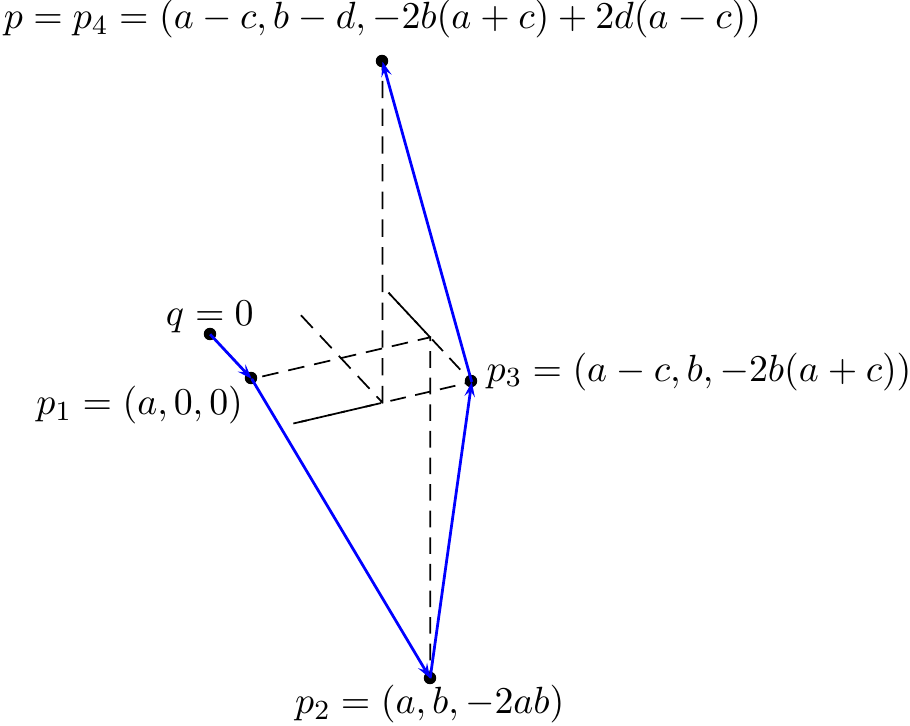}
\caption{A bang-bang path from $q=(0,0,0)$ to $p=(x,y,t)$}\label{f:PXYL}
\end{figure}
\begin{lemma} \label{L:PXYLpath} %
Any two points $p, q \in \Heis$ are connected by a path
$$\alpha=[p_0,p_1]\cup[p_1,p_2]\cup[p_2,p_3]\cup[p_3,p_4],$$
where $[p_0, p_1]$ and $[p_2, p_3]$ are $X$-line segments while $[p_1,p_2]$ and $[p_3, p_4]$ are $Y$-line segments, and
$\ellcc(\alpha)\le 5\sqrt{2} \dcc(p,q)\,.$
\end{lemma}

\begin{proof}
Without loss of generality, assume that $q=(0,0,0)$ and write $p$ as $p=(x,y,t)\neq (0,0,0)$. The desired points $p_0, \ldots, p_4\in \Heis$ must satisfy, for some $a,b,c,d$, the conditions
\begin{align*}
  p_0&=q=(0,0,0)\,,  \\
  p_1&=p_0*(a,0,0)=(a,0,0)\,,  \\
  p_2&=p_1*(0,b,0)=(a,b,-2ab)\,,  \\
  p_3&=p_2*(-c,0,0)=(a-c,b,-2b(a+c))\,,  \\
  p_4&=p_3*(0,-d,0)=(a-c,b-d,-2[b(a+c)+d(c-a)])\,.
\end{align*}
If $t=-2xy$, such conditions hold with
$$
a:=x\,, \quad b:=y \,, \quad\text{and}\quad c:=d:=0.
$$
If $t\ne-2xy$ they hold with
\begin{align*}
  b&:=\begin{cases}
        \displaystyle \bigl|\frac{t}2 + xy \bigr|^{1/2}  &\text{if $t+2xy>0$}\,, \\
        \displaystyle -\bigl|\frac{t}2 + xy \bigr|^{1/2}  &\text{if $t+2xy<0$}\,.
      \end{cases}
\end{align*}
$$
a:=x-(t+2xy)/(4b)\,, \quad c:=a-x\,, \quad\text{and}\quad d:=b-y.
$$
It is straightforward to verify that $p_4=p$.  

We now compute the length of $\alpha$, assuming that $t\ne-2xy$ so $b\ne 0$. Since
\begin{gather*}
  \frac{|t+2xy|}{4|b|}=\frac{b^2}{2|b|}=\frac12\,|b|
  \intertext{and}
  |b|\le \sqrt{|xy|}+\sqrt{|t|/2} \le \frac{|x|+|y|}2 + \frac{|t|^{1/2}}{\sqrt{2}}\,,
  \intertext{we have}
  |a|+|b| \le |x|+\frac{|t+2xy|}{4|b|}+|b| \le |x|+\frac32\,|b|
  \intertext{and hence}
  2\bigl(|a|+|b|\bigr) \le 2|x|+3|b| \le \frac{7|x|+3|y|}2 + \frac{3}{\sqrt{2}}\,|t|^{1/2}.
\end{gather*}
Thus
\begin{align*}
  \ellcc(\alpha)&=|a|+|b|+|c|+|d|=|a|+|b|+|a-x|+|b-y|  \\
                &\le 2(|a|+|b|)+|x|+|y| \le \frac{9|x|+5|y|}2 + \frac{3}{\sqrt{2}}\,|t|^{1/2}  \\
                &\le 5\bigl(|x|+|y|+ |t|^{1/2}  \bigr) \le 5\sqrt{2} \dhh(p,0) \le 5\sqrt{2} \dcc(p,0),
\end{align*}
where in the last line we make use of the Koranyi distance (or gauge distance, or Cygan distance) $\dhh$, defined by left-invariance and the property 
$$\dhh(p,0) = \sqrt[4]{(x^2+y^2)^2+t^2},$$
as well as the standard fact that $\dcc$ is the path metric associated to $\dhh$.
\end{proof}%

We now prove Theorem \ref{thm:main} by adjusting the path given in Lemma \ref{L:PXYLpath}.

\begin{proof}[Proof of Theorem \ref{thm:main}]

Let $A$ be a subset of $\Heis$ such that $\pi_x(A)$ and $\pi_y(A)$ are nowhere dense. Fix $p, q\in \comp A$ and connect them, using Lemma  \ref{L:PXYLpath}, by a bang-bang path $\gamma$ that decomposes as segments $[p_1^-, p_1^+]\cup[p_2^-, p_2^+]\cup[p_3^-,p_3^+]\cup[p_4^-,p_4^+]$, where the first and third segments are $X$-line segments, and the second and fourth segments are $Y$-line segments. 

If $\gamma \subset \comp A$, we are done. Otherwise, we fix $\epsilon>0$, to be determined later, and show how to build an alternate path from $p$ to $q$ as illustrated in Figure~\ref{f:proofSketch}. Namely, we perturb the four segments, denoting the new endpoints in a bold font, e.g.\ we perturb $[p_1^-, p_2^+]$ to produce $[\textbf{p}_1^-, \textbf{p}_1^+]\subset \comp A$. The new line segments will maintain the direction and length of the original segments, but will (generically) no longer touch at the endpoints or go through the points $p$ and $q$. However, if the perturbation is small we will still be able to join the appropriate points by short geodesics. We now show how to carry out this plan, apart from finding the final curve joining $\textbf{p}_4^+$ to $q$.

\begin{figure}[h]
\includegraphics[width=3in]{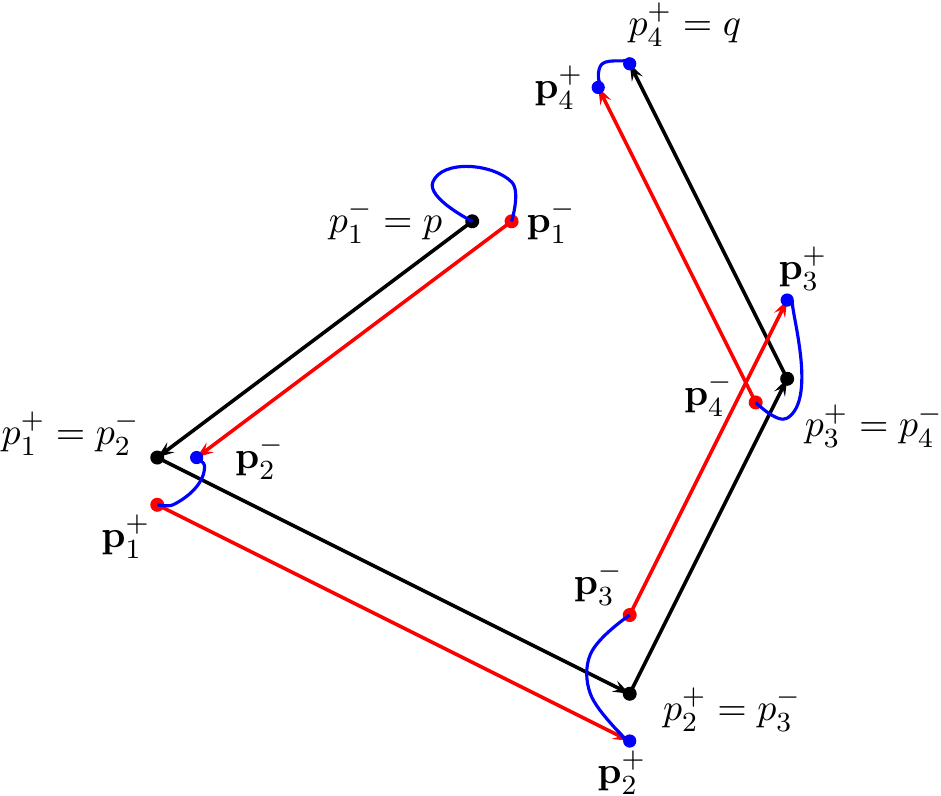}
\caption{To prove Theorem \ref{thm:main}, we perturb the (black) segments joining $p$ and $q$, producing (red) segments that avoid $A$. For sufficiently small perturbations, the endpoints can then be re-connected with (blue) geodesics outside of $A$.}\label{f:proofSketch}
\end{figure}

\begin{claim}\label{claim:perturbation}
There exist segments $[\textbf{p}_i^-,{\textbf p_i^+}]\subset \comp A$ for $i=1, \ldots, 4$ such that
\begin{enumerate}
\item $(p_i^-)^{-1}*p_i^+ = (\textbf p_i^-)^{-1}* \textbf{p}_i^+$,
\item $\dcc(p,\textbf{p}_1^-) < \epsilon$,
\item $\dcc(\textbf{p}_i^+,\textbf{p}_{i+1}^-)< \epsilon$ for $i=1, \ldots, 3$,
\end{enumerate}
Furthermore, there exist geodesic segments $\gamma_1, \ldots, \gamma_4 \subset \comp A$ of length at most $\epsilon$ so that $\gamma_1$ joins $p$ to $\textbf{p}_1^-$, and for $i=2, 3, 4$ the geodesic segment $\gamma_i$ joins $\textbf{p}_{i-1}^+$ to $\textbf{p}_i^-$.
\end{claim}

\begin{proof}
In view of the first condition, it  suffices to specify each $\textbf{p}_1^-, \ldots, \textbf{p}_4^- \in \comp A$ satisfying the second and third conditions.

We start by choosing $\textbf{p}_1^-$ near $p_1^-=p$. Noting that $A$ is closed, choose $0<\epsilon'<\epsilon$ so that $\Ball(p_1^-, \epsilon')\subset \comp A$.

The map $\pi_x$ is open and hence $\pi_x(\Ball(p_1^-, \epsilon'))$ is an open subset of the $yt$-plane. Thus, $\pi_x(\Ball(p_1^-, \epsilon'))\cap \pi_x(\comp A) \neq \emptyset$, and we may choose $\textbf{p}_1^-$ so that
$$
\pi_x(\textbf{p}_1^-) \in \pi_x(\comp A),
$$
$$
\dcc(\textbf{p}_1^-, p_1^-)<\epsilon'<\epsilon,
$$
and furthermore $\textbf{p}_1^-$ is connected to $p=p_1^-$ by a geodesic segment $\gamma_1\subset \Ball(p_1^-, \epsilon') \subset \comp A$.

We now choose $\textbf{p}_2^-$ near $\textbf{p}_1^+$. Noting that $A$ is closed, choose $0<\epsilon'<\epsilon$ so that $\Ball(p_1^+, \epsilon')\subset \comp A$. As above, $\pi_y(\comp A)$ is dense in the $xt$-plane, and $\pi_y(\Ball(p_1^+, \epsilon))$ is open in the $xt$-plane. Thus, $\pi_y(\Ball(\textbf{p}_1^+, \epsilon))\cap \pi_y(\comp A) \neq \emptyset$, and we may choose $\textbf{p}_2^-$ so that
$$
\pi_x(\textbf{p}_2^-) \in \pi_x(\comp A),
$$
$$
\dcc(\textbf{p}_2^-, \textbf{p}_1^+)<\epsilon'<\epsilon,
$$
and furthermore $\textbf{p}_2^-$ is connected to $p_1^+$ by a geodesic segment $\gamma_2\subset \Ball(\textbf{p}_1^+, \epsilon') \subset \comp A$.

We likewise obtain $\textbf{p}_3^-$ and $\gamma_3$ by perturbing $\textbf{p}_2^+$, and $\textbf{p}_4^-$ and $\gamma_4$ by perturbing $\textbf{p}_3^+$, with corresponding connecting geodesic segments $\gamma_3$ and $\gamma_4$.
\end{proof}

It remains to connect $\textbf{p}_4^+$ to $q$. Because $A$ is closed, there is some $\delta<\dcc(p,q)$ such that $\Ball(q, \delta)\subset \comp A$. We may choose a geodesic curve $\gamma_5$ of length at most $\delta$ joining $\textbf{p}_4^+$ to $q$ if we have $\dcc(\textbf{p}_4^+, q)<\delta$. Indeed, this can be arranged by a suitable choice of $\epsilon$.

\begin{claim}\label{claim:distanceToQ}
Then there exists $\epsilon_0$ depending only on $\delta$, $p$ and $q$ such that if $\epsilon<\epsilon_0$, then
$\dcc(q, \textbf{p}_4^+)< \delta.$
\end{claim}

\begin{proof}
Let us denote, for $h\in \Heis$, $\norm{h}:=\dcc(h,0)$. Our goal is equivalent to bounding $\norm{q*(\textbf{p}_4^+)^{-1}}$ by $\delta$. To this end, expand $q$ as
$$
q=p_3^+*[(p_4^-)^{-1}*p_4^+]
$$
and $\textbf{p}_4^+$ as
$$
\textbf{p}_4^+ = \textbf{p}_3^+ * [ (\textbf{p}_3^+)^{-1} * \textbf{p}_4^-] * [(p_4^-)^{-1}*p_4^+].
$$
Combining these, we obtain
\begin{align*}
q*(\textbf{p}_4^+)^{-1}&=
p_3^+*[(p_4^-)^{-1}*p_4^+]* [(p_4^-)^{-1}*p_4^+]^{-1} * [ (\textbf{p}_3^+)^{-1} * \textbf{p}_4^-]^{-1} * (\textbf{p}_3^+)^{-1}\\
								&=p_3^+* [ (\textbf{p}_3^+)^{-1} * \textbf{p}_4^-]^{-1} * (\textbf{p}_3^+)^{-1}.
\end{align*}
Since $\norm{(\textbf{p}_3^+)^{-1} * \textbf{p}_4^-}<\epsilon$, we would like to use the triangle inequality to conclude that  $q*(\textbf{p}_4^+)^{-1}$ and $p_3^+* (\textbf{p}_3^+)^{-1}$ are $\epsilon$-close, but we are hampered by the non-commutativity of $\Heis$. However, we may use the following well-known and easily-proven result.

\begin{lemma}\label{lemma:commutative}
Let $C>0$. For each $\eta>0$, there exists $\epsilon>0$ such that if $\norm{h_1}<C$ and $\norm{h_2}<\epsilon$, then $\dcc( h_1*h_2, h_2*h_1) < \eta.$
\end{lemma}

We fix $C>100 (\norm{p}+\norm{q})$ and apply Lemma \ref{lemma:commutative} with $\eta=\delta/8$ to obtain our $\epsilon_0$. Without loss of generality, we assume that also $\epsilon_0<\delta/8$.

We then have from Lemma \ref{lemma:commutative} and $\epsilon<\epsilon_0$ that
$$\dcc(p_3^+* [ (\textbf{p}_3^+)^{-1} *\textbf{p}_4^-]^{-1} * (\textbf{p}_3^+)^{-1}, p_3^+* (\textbf{p}_3^+)^{-1}* [ (\textbf{p}_3^+)^{-1} * \textbf{p}_4^-]^{-1}  )<\delta/8,$$
and so by the triangle inequality
$$\dcc(q*(\textbf{p}_4^+)^{-1}, p_3^+* (\textbf{p}_3^+)^{-1})< \delta/4.$$
We likewise obtain
$$\dcc( p_3^+* (\textbf{p}_3^+)^{-1},  p_{2}^+* (\textbf{p}_{2}^+)^{-1})<\delta/4$$
$$\dcc( p_2^+* (\textbf{p}_2^+)^{-1},  p_{1}^+* (\textbf{p}_{1}^+)^{-1})<\delta/4$$
and finally by the choice of $\textbf{p}_1^+$,
$$\norm{ p_{1}^+* (\textbf{p}_{1}^+)^{-1}} < \epsilon < \delta/8.$$
Combining these four estimates completes the claim.
\end{proof}

We now complete the proof of Theorem \ref{thm:main}. Let us assume that $\epsilon$ is smaller than both $\epsilon_0$ provided by Claim \ref{claim:distanceToQ} and $\dcc(p,q)$. Then Claim \ref{claim:perturbation} and Claim \ref{claim:distanceToQ} show how to choose a path $\gamma\subset A$ joining $p$ to $q$ as in Figure \ref{f:proofSketch}. Furthermore, the length of $\gamma$ is bounded above by the length of the original bang-bang path plus $5\dcc(p,q)$, as desired.
\end{proof}

\begin{remark}\label{rem:cor-main-alternate-proof}
Corollary \ref{cor:main} can also be proved by appealing to the ACL characterization of Sobolev functions on $\Heis$, and geometric properties of Sobolev removable sets. Let us sketch the argument. First, as already noted, the condition that $A$ is a null set for the cc-Hausdorff $3$-measure implies that $\pi_x(A)$ and $\pi_y(A)$ are also null sets for the cc-Hausdorff $3$-measure. Thus $A$ is removable for ACL functions, and hence also for Sobolev functions on $\Heis$. (See, for instance, Theorem 2.2 in \cite{tang:cr} for the relationship between the ACL and Sobolev conditions for Heisenberg source.) It follows that $A$ is a null set for Sobolev capacity, and consequently that $\Heis\setminus A$ is a Loewner space in the sense of Heinonen and Koskela \cite{hk:quasi}. By \cite[Theorem 3.13]{hk:quasi}, $\Heis\setminus A$ is quasiconvex.
\end{remark}

\section{A compact and totally disconnected set of cc-Hausdorff dimension three with nonquasiconvex complement} \label{S:ex} 

Here we prove Theorem \ref{thm:main-example}. We first recall the statement of that theorem.

\begin{thm}\label{thm:ex}
There exists a compact, totally disconnected set $A\subset \Heis$ of Hausdorff dimension $3$ such that $A^c=\Heis\setminus A$ is not quasiconvex.
\end{thm}

We first show how to deduce Theorem \ref{thm:ex} as a consequence of the following proposition.


\begin{prop}\label{claim:combineAn}
For each integer $n\geq 1$, there exists a set $A_n \subset \Heis$ such that the following conditions hold:
\begin{enumerate}
\item[(i)] $\comp A_n$ is not $n$-quasi-convex,
\item[(ii)] $A_n$ is compact and totally disconnected,
\item[(iii)] The Euclidean Hausdorff dimension of $A_n$ is at most $2$.
\end{enumerate}
\end{prop}

\begin{proof}[Proof of Theorem \ref{thm:ex}]
We suppose that sets $A_n$ as in the proposition exist, and show how to construct the desired set $A$.

For each $A_n$, the non-quasi-convexity condition implies that there exists a pair of points $h^-_n, h^+_n\in \comp A_n$ such that $$\dccOm{\comp{A_n}}(h^-_n, h^+_n)\geq n \dcc(h^-_n, h^+_n).$$
Because $A_n$ is compact, it is contained in some ball $\Ball(h_n, r_n)$. Enlarging $r_n$ if necessary, we may assume that the points $h^-_n$ and $h^+_n$ are also contained in $\Ball(h_n, r_n)$. Rescaling by a dilation map if necessary, we may furthermore assume that $r_n<n^{-2}/10$. Lastly, applying an (isometric left-multiplication) translation if necessary, we may assume $h_n=(0, 0, n^{-1})$. These normalizations do not affect the conditions on $A_n$: (i) is unchanged because $f\circ \delta_{r}$ is a cc-similarity, (ii) is unchanged because $f\circ \delta_{r}$ is a homeomorphism, and (iii) is unchanged because $f\circ \delta_{r'}$ is in fact a Euclidean affine mapping of $\R^3$.

Note that the balls $\Ball(h_n, r_n)$ containing the sets $A_n$ and points $h^\pm_n$ are disjoint, and let
$$
A:=\{(0,0,0)\}\cup (\cup_{n=1}^\infty A_n).
$$
The set $A$ is compact, totally disconnected, has Euclidean Hausdorff dimension at most $2$, and is not quasiconvex. By Corollary \ref{cor:main}, $A$ has cc-Hausdorff dimension at least $3$. By Theorem  \ref{DCT}, $A$ has cc-Hausdorff dimension at most $3$. Thus, $A$ verifies all the conditions of Theorem \ref{thm:ex}.
\end{proof}

It remains to prove Proposition \ref{claim:combineAn}. We begin with a lemma.

\begin{lemma}\label{claim:projection}
The complement of the box $\ABox_0:=[-10n, 10n]\times [-10n, 10n]\times[-1/2,1/2]$ is not $n$-quasi-convex. In particular, if $h^+:=(0,0,1)$ and $h^-:=(0,0,-1)$, then
$$
\dccOm{\comp{\ABox_0}}(h^+, h^-)\geq 20n>n\cdot \dcc(h^+, h^-).
$$
\end{lemma}

\begin{proof}
Let $\gamma$ be any rectifiable path in $\comp{\ABox_0}$ joining $h^+$ to $h^-$. The curve $\pi_t(\gamma)$ starts and terminates at the origin of $\R^2$, but must also exit the rectangle $\pi_t(\ABox_0)=[-10n, 10n]\times[-10n, 10n]$. Thus, $\pi_t(\gamma)$ has Euclidean length at least $20n$. Since the image of a cc-rectifiable curve under $\pi_t$ is a Euclidean rectifiable curve of the same length, the cc-length of $\gamma$ is also at least $20n$. Since $\dcc(h^+, h^-)=\sqrt{2\pi}$, we have the desired bound on $\dccOm{\comp {\ABox_0}}(h^+, h^-)$.
\end{proof}

Note that in the proof of Lemma \ref{claim:projection}, we only measure length of curves after projecting to $\R^2$. Motivated by this observation, we make the following definition.

\begin{defi}
The \emph{$\pi_t$-length} of a continuous curve $\gamma \subset \Heis$ is the \emph{taxicab ($\ell^1$) length} of its projection $\pi_t \gamma$ to $\R^2$.

Given points $h_1, h_2$ in a domain $\Omega \subset \Heis$, the \emph{$\pi_t$-distance in $\Omega$} between $h_1$ and $h_2$, denoted $\dpi{\Omega}(h_1, h_2)$, is  the infimum of $\pi_t$-lengths of all continuous curves $\gamma \subset \Omega$ joining $h_1$ and $h_2$.
\end{defi}

\begin{remark}
Note that while $\dpi{\Heis}(h_1, h_2)$ is just the $\ell^1$ distance between $\pi_t(h_1)$ and $\pi_t(h_2)$, the $\pi_t$-distance $\dpi{\Omega}$ on a subset $\Omega\subset \Heis$ does not have such a simple expression.
\end{remark}

Critically, $\pi_t$-length does not distinguish between the Heisenberg and Euclidean geometries on $\R^3$. For our purposes, it is more flexible than either Heisenberg or Euclidean length. In particular, if $\gamma\subset \R^3$ is a curve, then:
\begin{enumerate}
\item if $\gamma$ is cc-rectifiable, then its cc-length agrees with the $\ell^2$ length of $\pi_t(\gamma)$, which in turn is bounded below by the $\pi_t$-length of $\gamma$ divided by $\sqrt{2}$,
\item if $\gamma$ is Euclidean rectifiable, then its Euclidean length is bounded below the $\ell^2$ length of $\pi_t(\gamma)$, which in turn is bounded below by the $\pi_t$-length of $\gamma$ divided by $\sqrt{2}$,
\item $\gamma$ may have finite $\pi_t$-length even if it is not cc rectifiable,
\item $\gamma$ may have finite $\pi$-length even if it is not Euclidean rectifiable,
\item vertical subsegments in $\gamma$ do not contribute to its $\pi_t$-length.
\end{enumerate}

\begin{lemma}\label{claim:dpiReduction}
Let $A \subset \ABox_0$. If $\dpi{\comp{A}}(h^+, h^-) \geq \dpi{\comp{\ABox_0}}(h^+, h^-)$, then the complement of $A$ is not $n$-quasiconvex.
\end{lemma}

\begin{proof}
Under the stated assumptions,
\begin{align*}\dccOm{\comp{A}}(h^+, h^-) &\geq
\frac{1}{\sqrt 2}\dpi{\comp{A}}(h^+, h^-) \geq \frac{1}{\sqrt 2} \dpi{\comp{\ABox_0}}(h^+, h^-)\\&\geq \frac{20n}{\sqrt{2}} > n \sqrt{2\pi} = n \dcc(h^+, h^-).  \qedhere
\end{align*}
\end{proof}

The preceding discussion understood, our task no longer involves quasiconvexity in either $\Heis$ or in Euclidean $\R^3$; we need only consider the $\pi_t$-lengths of curves. The proof of Theorem \ref{thm:ex} is complete once we recall the following construction from \cite{hh:quasiconvexity}.

\begin{thm}[Hakobyan--Herron]\label{thm:HHEx}
There exists a set $A_n \subset \ABox_0$ that satisfies:
\begin{enumerate}
\item \label{metriccondition} $\dpi{\comp{A_n}}(h^+, h^-) \geq \dpi{\comp{\ABox_0}}(h^+, h^-)$,
\item \label{measurecondition} the Euclidean Hausdorff dimension of $A_n$ is equal to $2$,
\item \label{topologicalcondition} $A_n$ is compact and totally disconnected.
\end{enumerate}
\end{thm}

For the reader's convenience, we sketch the main ideas in the proof of Theorem \ref{thm:HHEx}.

\begin{proof}[Sketch of proof of Theorem \ref{thm:HHEx}]
The box $\ABox_0$ certainly satisfies condition \eqref{metriccondition}, but it does not satisfy the remaining conditions. Towards the latter, let $\ABox_1 \subset \ABox_0$ be the union of $24$ smaller boxes $\ABox_{(1)}, \ldots, \ABox_{(24)}$, arranged as in Figure \ref{fig:boxes}.

\begin{figure}
\includegraphics[width=.5\textwidth]{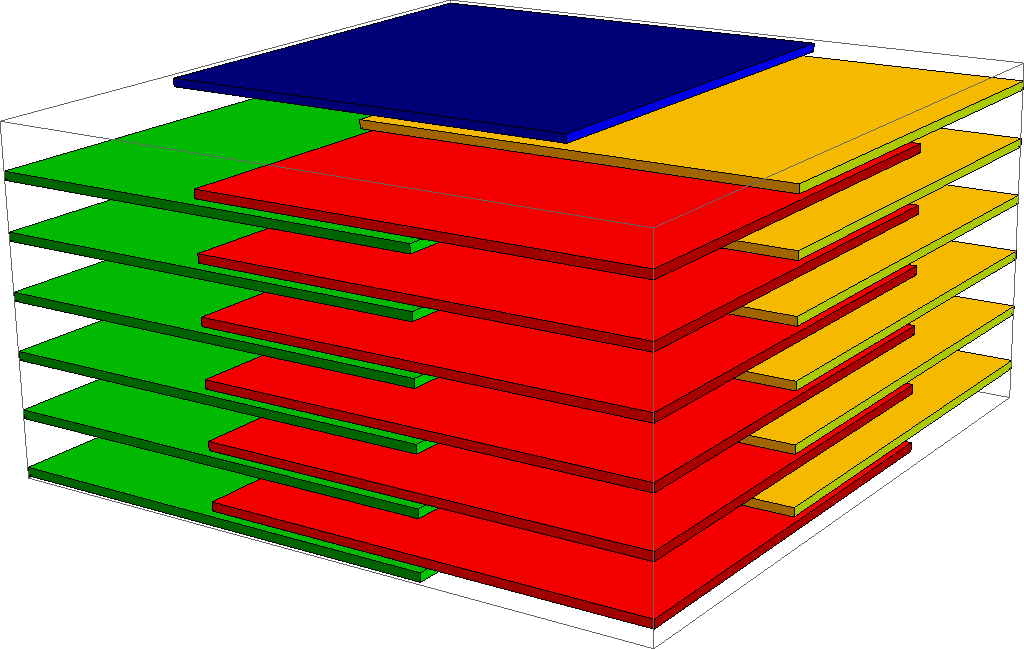}
\label{fig:boxes}\caption{A maze of disjoint boxes $\ABox_{(1)}, \ldots, \ABox_{(24)}$ inside $\ABox_0$}
\end{figure}
It is easy to see that it is inefficient for a curve to enter $\ABox_0$ but avoid $\ABox_1$. Indeed, consider $h_1, h_2 \in \partial \ABox_0$. If $h_1$ and $h_2$ are on the same face of $\partial \ABox_0$, it is obvious that any path $\gamma$ can be retracted to the boundary without increasing its $\pi_t$-length. Points on different faces can be handled in the same way, except for the case of one point on the top face and the other on the bottom face. In that case, fix a path $\gamma$ between $h_1$ and $h_2$ in the complement of $\ABox_1$, and distinguish two subcases.
If $\pi_t(\gamma)$ leaves the open rectangle $\interior \pi_t (\ABox_0)$, then it is easy to retract $\gamma$ to $\partial \ABox_0$ without increasing its $\pi_t$-length.
If $\pi_t(\gamma)$ doesn't leave the open rectangle $\interior \pi_t (\ABox_0)$, then it must dodge the rectangles comprising $\ABox_1$. A simple argument tracking the different paths that $\gamma$ can take shows that the $\pi_t$-length of $\gamma$ must be at least $40n$, provided $\pi_t(\ABox_{(1)}), \ldots, \pi_t(\ABox_{(24)})$ have sufficient pairwise overlap. In this case, one constructs a shorter path between $p_1$ and $p_2$ directly along the boundary $\partial \ABox_0$.

Continuing inductively, construct inside each box $\ABox_{(i_1, \ldots, i_j)}$ of level $j$ 24 boxes  $\ABox_{(i_1, \dots, i_j, 1)}, \ldots, \ABox_{(i_1, \dots, i_j, 24)}$ of level $j+1$ and denote by $\ABox_j$ the union of all boxes of level $j$. This gives a sequence of nested compact sets $\ABox_0 \supset \ABox_1 \supset \cdots$. Define $A_n := \cap_j \ABox_j$; this set clearly satisfies condition \eqref{topologicalcondition}.

To verify condition \eqref{metriccondition}, let $\gamma$ be a curve joining $h^+$ and $h^-$ in $\comp{A_n}	$. It suffices to show that it is longer than some curve in the complement of $\ABox_0$. By compactness, $\gamma$ intersects only finitely many boxes $\ABox_{(i_1, \ldots, i_j)}$ and in particular some box $\ABox_{(i_1, \ldots, i_j)}$ of maximal level among these. But as above, we may adjust $\gamma$ to not enter $\ABox_{(i_1, \ldots, i_j)}$ without increasing its length. Proceeding inductively, we may push $\gamma$ out of all of the sub-boxes of $\ABox_0$ and eventually out of $\ABox_0$ itself, without increasing its $\pi_t$-length, as desired.

Finally, since $\pi_t(A_n)$ is a square, $A_n$ has Euclidean Hausdorff dimension at least $2$. A careful choice of the boxes used to construct $A_n$ and a box-counting-dimension argument show that $A_n$ must have dimension exactly equal to $2$. This verifies the last condition \eqref{measurecondition} and completes the proof.
\end{proof}

\bibliographystyle{acm}
\bibliography{bib}

\end{document}